\documentclass[12pt]{article}
\usepackage{amssymb, amsmath, amsthm, amsfonts,xcolor, enumerate}
\usepackage{tikz}
\usetikzlibrary{patterns}
\usetikzlibrary{shapes}
\usepackage{fullpage}
\usepackage{hyperref}
\newtheorem{theorem}{Theorem}[section]
\newtheorem{claim}[theorem]{Claim}
\newtheorem{cor}[theorem]{Corollary}
\newtheorem{fact}[theorem]{Fact}
\newtheorem{prop}[theorem]{Proposition}
\newtheorem{conj}[theorem]{Conjecture}
\newtheorem{ques}[theorem]{Question}
\newtheorem{lemma}[theorem]{Lemma}
\theoremstyle{definition}

\theoremstyle{definition}

\theoremstyle{definition}

\theoremstyle{definition}

\theoremstyle{definition}

\theoremstyle{definition}

\theoremstyle{definition}

\theoremstyle{definition}

\newcommand\ex{\ensuremath{\mathrm{ex}}}

\newcommand{\ep}{\varepsilon}

\newcommand{\Om}{\Omega}

\newcommand{\al}{\alpha}
\newcommand{\de}{\delta}

\newcommand{\ga}{\gamma}
\newcommand{\Ga}{\Gamma}
\newcommand{\De}{\Delta}

\newcommand{\cH}{\mathcal{H}}

\newcommand{\cG}{\mathcal{G}}

\newcommand{\cB}{\mathcal{B}}

\newcommand{\cC}{\mathcal{C}}
\newcommand{\Ho}{\H{o}}

\newcommand{\bZ}{\ensuremath{\mathbb{Z}}}
\newcommand{\bP}{\ensuremath{\mathbb{P}}}
\newcommand{\pri}{\mathbf{P}_{\le n}}
\newcommand{\bE}{\ensuremath{\mathbb{E}}}

\title{The number of subsets of integers with no $k$-term arithmetic progression}
\author{
J\'ozsef Balogh
 \thanks{Department of Mathematical Sciences,
 University of Illinois at Urbana-Champaign, Urbana, Illinois 61801, USA. Email: {\tt
jobal@math.uiuc.edu}. Research is partially supported by NSA Grant H98230-15-1-0002, NSF DMS-1500121 and Arnold O. Beckman Research Award (UIUC Campus Research Board 15006).}
 \quad\quad
 Hong Liu
 \thanks{Mathematics Institute and DIMAP, University of Warwick, Coventry, CV4 7AL, UK.  Email: {\tt h.liu.9@warwick.ac.uk}. This research was done while HL was at University of Illinois at Urbana-Champaign.}
 \quad\quad
 Maryam Sharifzadeh
 \thanks{Department of Mathematical Sciences,
 University of Illinois at Urbana-Champaign, Urbana, Illinois 61801, USA. Email: {\tt
sharifz2@illinois.edu}.}
}

\begin{document}
\maketitle
\begin{abstract}
Addressing a question of Cameron and Erd\Ho s, we show that, for infinitely many values of $n$, the number of subsets of $\{1,2,\ldots, n\}$ that do not contain a $k$-term arithmetic progression is at most $2^{O(r_k(n))}$, where $r_k(n)$ is the maximum cardinality of a subset of $\{1,2,\ldots, n\}$ without a $k$-term arithmetic progression. This bound is optimal up to a constant factor in the exponent. For all values of $n$, we prove a weaker bound, which is nevertheless sufficient to transfer the current best upper bound on $r_k(n)$ to the sparse random setting. To achieve these bounds, we establish a new supersaturation result, which roughly states that sets of size $\Theta(r_k(n))$ contain superlinearly many $k$-term arithmetic progressions.

For integers $r$ and $k$, Erd\Ho s asked whether there is a set of integers $S$ with no $(k+1)$-term arithmetic progression, but such that any $r$-coloring of $S$ yields a monochromatic $k$-term arithmetic progression. Ne\v{s}et\v{r}il and R\"odl, and independently Spencer, answered this question affirmatively. We show the following density version: for every $k\ge 3$ and $\de>0$, there exists a reasonably dense subset of primes $S$ with no $(k+1)$-term arithmetic progression, yet every $U\subseteq S$ of size $|U|\ge\de|S|$ contains a $k$-term arithmetic progression.

Our proof uses the hypergraph container method, which has proven to be a very powerful tool in extremal combinatorics. The idea behind the container method is to have a small certificate set to describe a large independent set. We give two further applications in the appendix using this idea.
\end{abstract}


\section{Introduction}\label{sec-intro}
Enumerating discrete objects in a given family with certain properties is one of the most fundamental problems in extremal combinatorics. In the context of graphs, this was initiated by Erd\Ho s, Kleitman and Rothschild~\cite{ekr76}  who studied the family of triangle-free graphs. For recent developments, see e.g.~\cite{MTF,MS} and the references therein. In this paper, we investigate counting problems in the arithmetic setting. In this direction, one of the major open problems, raised by Cameron and Erd\Ho s~\cite{CE}, was to prove that the number of sum-free sets\footnote{A set $S$ is \emph{sum-free}, if for any $x,y\in S$, $x+y\notin S$.} in $\{1,2,\ldots,n\}$ is $O(2^{n/2})$. This conjecture was proven independently by Green~\cite{G-CE} and Sapozhenko~\cite{sap}. See also~\cite{BLST,BLST2} for the proof of   another conjecture of Cameron and Erd\Ho s~\cite{CE-max} concerning the family of maximal sum-free sets.

Our main results are given in Section~\ref{sec-1-1}: the first is on counting subsets of integers without an arithmetic progression of fixed length (see Theorem~\ref{thm-main}); the second is a new supersaturation result for arithmetic progressions (see Theorem~\ref{thm-supsat}). In Section~\ref{sec-1-2}, we show the existence of a set of primes without a $(k+1)$-term arithmetic progression which however is very rich in $k$-term arithmetic progressions (see Theorem~\ref{thm-prime-density}\footnote{A similar result is discussed in the Appendix as well as an additional problem about sumsets.}).

\subsection{Enumerating sets with no $k$-term arithmetic progression}\label{sec-1-1}
A subset of $[n]:=\{1, 2,\ldots,n\}$ is $k$\emph{-AP-free} if it does not contain a $k$-term arithmetic progression. Denote by $r_k(n)$ the maximum size of a $k$-AP-free subset of $[n]$. Cameron and Erd\Ho s~\cite{CE} raised the following question: How many subsets of $[n]$ do not contain a $k$-term arithmetic progression?  In particular, they asked the following question.
\begin{ques}[Cameron-Erd\Ho s]\label{conj-CE}
Is it true that the number of $k$-AP-free subsets of $[n]$ is 
$$2^{(1+o(1))r_k(n)}?$$
\end{ques}

Since every subset of a $k$-AP-free set is also $k$-AP-free, one can easily obtain $2^{r_k(n)}$ many $k$-AP-free subsets of $[n]$. In fact, Cameron and Erd\Ho s~\cite{CE} slightly improved this obvious lower bound: writing $R_k(n)$ for the number of $k$-AP-free subsets of $[n]$, they proved that
\begin{eqnarray}\label{eq-false}
\limsup_{n\rightarrow \infty} \frac{R_k(n)}{2^{r_k(n)}}=\infty.
\end{eqnarray}

Until recently, the only progress on the upper bound in the last 30 years was improving the bounds on $r_k(n)$. Then Balogh, Morris and Samotij~\cite{BMS}, and independently Saxton and Thomason~\cite{ST}, proved the following: for any $\beta>0$ and integer $k\ge 3$, there exists $C>0$ such that for $m\ge Cn^{1-1/(k-1)}$, the number of $k$-AP-free $m$-sets in $[n]$ is at most ${\beta n\choose m}$. This deep counting result implies the sparse random analogue of Szemer\'edi's theorem~\cite{Sz2} which was proved earlier by Conlon and Gowers~\cite{CG} and independently by Schacht~\cite{Sch}. However, this bound is far from settling Question~\ref{conj-CE}.

One of the reasons for the difficulty in finding good upper bounds on $R_k(n)$ is our limited understanding of $r_k(n)$. Indeed, despite much effort, the gap between the current known lower and upper bounds on $r_3(n)$ is still rather large; closing this gap remains one of the most difficult problems in additive number theory. For the lower bound on $r_3(n)$, the celebrated construction of Behrend~\cite{Beh} shows that
\begin{eqnarray*}
r_3(n)=\Omega\left( \frac{n}{2^{2\sqrt{2}\sqrt{\log_2n}}\cdot \log^{1/4}n}\right).
\end{eqnarray*}
This was recently improved by Elkin~\cite{Elk} by a factor of $\sqrt{\log n}$, see also Green and Wolf~\cite{GW}. Roth~\cite{Ro} gave the first non-trivial upper bound on $r_3(n)$, followed by the improvements of Heath-Brown~\cite{HB}, Szemer\'edi~\cite{Sz}, Bourgain~\cite{Bou} and a recent breakthrough of Sanders~\cite{S}. The current best bound is due to Bloom~\cite{Bloom}: 
\begin{eqnarray}\label{eq-sanders}
r_3(n)=O\left(\frac{n(\log\log n)^4}{\log n}\right).
\end{eqnarray}
For $k\ge 4$  there exist $c_k,c_k'>0$ such that
\begin{eqnarray}\label{eq-behrend}
\frac{n}{2^{c_k(\log n)^{1/(k-1)}}}\le r_k(n)\le \frac{n}{(\log\log n)^{c_k'}},
\end{eqnarray}
where the lower bound is due to Rankin~\cite{Ran} and the upper bound is by Gowers~\cite{Gow1,Gow2}.

Notice that, using the lower bound in~\eqref{eq-behrend}, we obtain the following trivial upper bound for $R_k(n)$: 
$$R_k(n)\le \sum_{i=0}^{r_k(n)}{n\choose i}<2{n\choose r_k(n)}<2\left(\frac{en}{r_k(n)}\right)^{r_k(n)}=2^{O\left(r_k(n)\cdot(\log n)^{\frac{1}{k-1}}\right)}.$$
We show that, for infinitely many $n$, the $(\log n)^{\frac{1}{k-1}}$ term in the exponent is not needed, i.e.~our result is optimal up to a constant factor in the exponent.
\begin{theorem}\label{thm-main}
The number of $k$-AP-free subsets of $[n]$ is  $2^{O(r_k(n))}$ for infinitely many values of $n$.
\end{theorem}

An immediate corollary of Theorem~\ref{thm-main} is the following.
\begin{cor}\label{cor-dense}
For every $\ep>0$, there exists a constant $b>0$ such that the following holds. Let $A(b)\subseteq \mathbb{Z}$ consist of all integers $n$ such that the number of $k$-AP-free subsets of $[n]$ is at most $2^{b\cdot r_k(n)}$. Then 
$$\limsup_{n\rightarrow \infty} \frac{|A(b)\cap [n]|}{n} \ge 1-\ep.$$
\end{cor}

Enumerating discrete structures with certain local constraints is a central topic in combinatorics. Theorem~\ref{thm-main} is the first such result in which the order of magnitude of the corresponding extremal function is not known. 

It is also worth mentioning that two other natural conjectures of Erd\Ho s are false: it was conjectured that the number of Sidon sets\footnote{A set $A\subseteq [n]$ is a \emph{Sidon} set if there do not exist distinct $a,b,c,d\in A$ such that $a+b=c+d$.} in $[n]$, denoted by $S(n)$, is $2^{(1+o(1))s(n)}$, where $s(n)$ denotes the size of a maximum Sidon set. However, it is known that $2^{1.16s(n)}\le S(n)=2^{O(s(n))}$, where the lower bound is by Saxton and Thomason~\cite{ST} and the upper bound is by Kohayakawa, Lee, R{\"o}dl and Samotij~\cite{Sidon} (see also~\cite{ST}). Another conjecture of Erd\Ho s states that the number of $C_6$-free\footnote{Denote by $C_k$ the cycle of length $k$. Given a graph $H$, a graph $G$ is $H$-\emph{free} if $G$ does not contain $H$ as a subgraph. Denote by $\ex(n,G)$ the maximum number of edges a $G$-free graph can have.} graphs on vertex set $[n]$, denoted by $H(n)$, is $2^{(1+o(1))\ex(n,C_6)}$. However, $2^{1.0007\ex(n,C_6)}\le H(n)=2^{O(\ex(n,C_6))}$, where the lower bound is by Morris and Saxton~\cite{MS} and the upper bound is by Kleitman and Wilson~\cite{KW}. In view of these examples and~\eqref{eq-false}, it is not inconceivable that the answer to Question~\ref{conj-CE} is no.

For all values of $n$, we obtain the following weaker counting estimate, which is nevertheless sufficient to improve  previous transference theorems for Szemer\'edi's theorem, in particular implies Corollary~\ref{cor-cond}.
\begin{theorem}\label{prop-cond}
If $r_k(n)\le \frac{n}{h(n)}$, where $h(n)\le (\log n)^c$ for some $c>0$, then the number of $k$-AP-free subsets of $[n]$ is at most $2^{O(n/h(n))}$. Furthermore, for any $\ga>0$, there exists $C=C(k,c,\ga)>0$ such that for any $m\ge n^{1-\frac{1}{k-1}+\ga}$, the number of $k$-AP-free $m$-subsets of $[n]$ is at most
$${Cn/h(n)\choose m}.$$
\end{theorem}

Theorem~\ref{prop-cond} improves the counting result of Balogh-Morris-Samotij~\cite{BMS} and Saxton-Thomason~\cite{ST} with a slightly weaker bound on $m$. We say that a set $A\subseteq \mathbb{N}$ is $(\delta, k)$-\emph{Szemer\'edi} if every subset of $A$ of size at least $\de|A|$ contains a $k$-AP. Denote by $[n]_p$ the $p$-random subset of $[n]$, where each element of $[n]$ is chosen with probability $p$ independently of others. As mentioned earlier, the counting result of~\cite{BMS} and~\cite{ST} implies the following sparse analogue of Szemer\'edi's theorem, which was only recently proved by a breakthrough transference theorem of Conlon and Gowers~\cite{CG} and Schacht~\cite{Sch}: For any constant $\de>0$ and integer $k\ge 3$, there exists $C>0$, such that almost surely $[n]_p$ is $(\de, k)$-Szemer\'edi for $p\ge Cn^{-\frac{1}{k-1}}$. As an easy corollary of Theorem~\ref{prop-cond}, we obtain the following sharper version, in which $\de$ could be taken as a function of $n$. In fact, it transfers the current best bounds on $r_k(n)$ of Bloom~\cite{Bloom} and Gowers~\cite{Gow1,Gow2} to the random setting. Proving Corollary~\ref{cor-cond} from Theorem~\ref{prop-cond} is similar as in~\cite{BMS} and~\cite{ST}, thus we omit the proof here. We remark that the bound on $p$ is optimal up to the additive error term $\ga$ in the exponent.
\begin{cor}\label{cor-cond}
If $r_k(n)\le \frac{n}{h(n)}$, where $h(n)\le (\log n)^c$ for some constant  $c>0$, then for any $\ga>0$, there exists $C=C(k,c,\ga)>0$ such that the following holds. If $p_n\ge n^{-\frac{1}{k-1}+\ga}$ for all sufficiently large $n$, then
$$\lim_{n\rightarrow\infty}\bP\left([n]_{p_n} \mbox{ is } \left(\frac{C}{h(n)}, k\right)\mbox{-Szemer\'edi}\right)=1.$$
\end{cor}
Combining the upper bounds in~\eqref{eq-sanders} and~\eqref{eq-behrend} with Corollary~\ref{cor-cond}, for some $C>0$, we have that almost surely $[n]_p$ is $\left(\frac{C(\log\log n)^4}{\log n}, 3\right)$-Szemer\'edi for $p\ge n^{-\frac{1}{2}+o(1)}$; and for $k\ge 4$ that almost surely $[n]_p$ is $\left(\frac{C}{(\log\log n)^{c_k'}}, k\right)$-Szemer\'edi for $p\ge n^{-\frac{1}{k-1}+o(1)}$.

The proof of Theorem~\ref{thm-main} uses the hypergraph container method, developed by Balogh, Morris and Samotij~\cite{BMS}, and independently by Saxton and Thomason~\cite{ST}. In order to apply the hypergraph container method, we need a supersaturation result. Supersaturation problems are reasonably well-understood if the extremal family is of positive density. For example, the largest sum-free subset of $[n]$ has size $\lceil n/2\rceil$, while any set of size $(\frac{1}{2}+\ep)n$ has $\Omega(n^2)$ triples satisfying $x+y=z$ (see~\cite{G-Removal}). In the context of graphs, the Erd\Ho s-Stone theorem gives\footnote{The \emph{chromatic number} of $G$, denoted by $\chi(G)$, is the minimum number of colors needed to color the vertices of $G$ such that no two adjacent vertices receive the same color.} $\ex(n,G)=(1-\frac{1}{\chi(G)-1}+o(1))\frac{n^2}{2}$, while any $n$-vertex graph with $(1-\frac{1}{\chi(G)-1}+\ep)\frac{n^2}{2}$ edges contains $\Omega(n^{|V(G)|})$ copies of $G$. However, the degenerate case is significantly harder. Indeed, a famous unsolved conjecture of Erd\H os and Simonovits~\cite{ErdSim} in extremal graph theory asks whether an $n$-vertex graph with $\ex(n,C_4)+1$ edges has at least two copies of $C_4$. 

For arithmetic progressions, the supersaturation result concerned only sets of size linear in $n$, more precisely   Varnavides~\cite{V} proved that any subset of $[n]$ of size $\Omega(n)$ has $\Omega(n^2)$ $k$-APs.  (see also~\cite{CS}). 
More recently,
Croot and Sisask~\cite{CS} proved  a nice formula, which is unfortunately not helping when $|A|\le O(r_k(n))$
and $r_k(n)\ll n/f(n)$ where $f(n)$ is a polylogarithmic function. Their formula is that for every $A\subset[n]$, and every $1\le M\le n$, the number of $3$-APs in $A$ is at least
$$\left( \frac{|A|}{n}-\frac{r_3(M)+1}{M}\right)\cdot \frac{n^2}{M^4}.$$
 We need a  supersaturation for sets of size $\Theta(r_k(n))$. Our second main result shows that the number of $k$-APs in any set $A$ of size constant factor times greater  than $r_k(n)$ is superlinear in $n$.
\begin{theorem}\label{thm-supsat}
Given $k\ge 3$, there exists a constant $C'=C'(k)>0$ and an infinite sequence $\{n_i\}_{i=1}^{\infty}$, such that the following holds. For any $n\in\{n_i\}_{i=1}^{\infty}$ and any $A\subseteq [n]$ of size $C'r_k(n)$, the number of $k$-APs in $A$ is at least 
$$\log^{3k-2}n\cdot \left(\frac{n}{r_k(n)}\right)^{k-1}\cdot n.$$
\end{theorem}

\subsection{Arithmetic progressions in the primes}\label{sec-1-2}
The study of arithmetic progressions in the set of primes has witnessed great advances in the last decade. Extending the seminal result of Szemer\'edi~\cite{Sz2}, Green and Tao in their landmark paper~\cite{GT} proved that any subset of the primes with positive relative density contains arbitrarily long arithmetic progression. In fact, they showed the following supersaturation version. Denote by $\pri$ the set of primes which are at most $n$. Then  the number of $k$-APs in any subset $U\subseteq\pri$ with $|U|=\Omega(|\pri|)$ is $\Theta(n^2/\log^kn)$. We are interested in the following question: does there exist a subset of primes that is $(k+1)$-AP-free, yet any subset of it with positive density contains a $k$-AP? A priori, it is not even clear whether such a set exists in $\bZ$, since intuitively a $(k+1)$-AP-free set is unlikely to be rich in $k$-APs. It is worth mentioning that Erd\Ho s~\cite{Erd} asked whether, for every integer $r$, there is a set of integers with no $(k+1)$-AP, but any $r$-coloring of it yields a monochromatic $k$-AP. Spencer~\cite{Sp} proved the existence of such a set and Ne\v{s}et\v{r}il and R\"odl~\cite{NR-vdw} constructed such a set. The question raised above is a strengthening of Erd\Ho s' in two aspects: it is a density version and asks for a set of primes. Our next result gives an affirmative answer to this question.

\begin{theorem}\label{thm-prime-density}
For any $\de> 0$ and $k\ge 3$, there exists a set of primes $S\subseteq \pri$ of size $n^{1-1/k-o(1)}$ such that $S$ is $(k+1)$-AP-free and $(\de,k)$-Szemer\'edi.
\end{theorem}

One might attempt to find a set of integers with the desired properties and then apply it to a very long arithmetic progression in the primes guaranteed by the Green-Tao theorem. However the subset of primes obtained in this way would be extremely sparse in $\pri$. To obtain a fairly dense subset in $\pri$ with the desired properties given Theorem~\ref{thm-prime-density}, we will instead do things in the ``reverse'' order. We first use the supersaturation version of the Green-Tao theorem and the container method to get the following counting result, which together with a standard application of the probabilistic method (see for similar and earlier applications of   Nenadov and Steger~\cite{NS} and of R\"odl, Rucinski and Schacht~\cite{RRS})   will establish Theorem~\ref{thm-prime-density}.

\begin{theorem}\label{thm-prime-counting}
For any $\beta> 0$, $\ga>0$ and $k\ge 3$, the number of $k$-AP-free $m$-subsets of $\pri$ with $m\ge n^{1-\frac{1}{k-1}+\ga}$ is at most
$${\beta|\pri|\choose m}.$$
Consequently, the number of $k$-AP-free subsets of $\pri$ is at most $2^{o(|\pri|)}$.
\end{theorem}

We omit the proof of Theorem~\ref{thm-prime-counting} since it follows along the same line as that of Theorem~\ref{prop-cond}. The only difference is that, for the supersaturation, we use the Green-Tao theorem instead of Lemma~\ref{lem-supsat}. Similarly to Theorem~\ref{prop-cond}, Theorem~\ref{thm-prime-counting} implies the following sparse random analogue of the Green-Tao theorem.
\begin{cor}\label{cor-prime}
For any $\de>0$ and $\ga>0$, if $p_n\ge n^{-\frac{1}{k-1}+\ga}$ for all sufficiently large $n$ and $S_n$ is a $p$-random subset of $\pri$, then
$$\lim_{n\rightarrow\infty}\bP\left(S_n \mbox{ is } \left(\de, k\right)\mbox{-Szemer\'edi}\right)=1.$$
\end{cor}


\noindent\textbf{Organization.} The rest of the paper will be organized as follows. In Section~\ref{sec-tool}, we introduce the hypergraph container method and some lemmas needed for proving supersaturation. In Section~\ref{sec-main}, we prove our main result, Theorem~\ref{thm-main}, and also Corollary~\ref{cor-dense}, Theorem~\ref{prop-cond}, Theorem~\ref{thm-supsat} and  Theorem~\ref{thm-prime-density}.

\medskip

\noindent\textbf{Notation.} We write $[a,b]$ for the interval $\{a, a+1,\ldots, b\}$ and $[n]:=[1,n]$. Given a set $A\subseteq [n]$, denote by $\Ga_k(A)$ the number of $k$-APs in $A$. Denote by $\min(A)$ the smallest element in $A$. We write $\log$ for logarithm with base 2. Throughout the paper we omit floors and ceilings where they are not crucial.


%
%
%


\section{Preliminaries}\label{sec-tool}
In the next subsection, we present the hypergraph container theorem and derive a version tailored for arithmetic progressions. We then prove some supersaturation results needed for the proof of Theorem~\ref{thm-supsat} in Section~\ref{sec-supsat}.

To see how they work, we give a quick overview of the proof of Theorem~\ref{thm-main}. We first apply the hypergraph container theorem (Corollary~\ref{cor-container}) to obtain a small collection of containers covering all $k$-AP-free sets in $[n]$, each of these containers having only few copies of $k$-APs. Then we apply the supersaturation result (Theorem~\ref{thm-supsat}) to show that every container necessarily has to be small in size ($O(r_k(n))$), from which our main result follows.

\subsection{The hypergraph container theorem}\label{subsec-container}

An $r$-uniform hypergraph $\cH=(V,E)$ consists of a vertex set $V$ and an edge set $E$, in which every edge is a set of $r$ vertices in $V$. An \emph{independent} set in $\cH$ is a set of vertices inducing no edge in $E$. The \emph{independence number} $\al(\cH)$ is the maximum cardinality of an independent set in $\cH$. Denote by $\chi(\cH)$ the \emph{chromatic number} of $\cH$, i.e., the minimum integer $\ell$, such that $V(\cH)$ can be colored by $\ell$ colors with no monochromatic edge.

Many classical theorems in combinatorics can be phrased as statements about independent sets in a certain auxiliary hypergraph. For example, the celebrated theorem of Szemer\'edi~\cite{Sz2} states that for $V(\cH)=[n]$ and $E(\cH)$ consisting of all $k$-term arithmetic prog\-ressions in $[n]$, $\al(\cH)=o(n)$. The cornerstone result of Erd\Ho s and Stone~\cite{Er-St} in extremal graph theory characterizes the structure of all maximum independent sets in $\cH$, where $V(\cH)$ is the edge set of $K_n$ and $E(\cH)$ is the edge set of copies of some fixed graph $G$. 

We will use the method of hypergraph containers for the proof of Theorem~\ref{thm-main}. This power\-ful method was recently introduced independently by Balogh, Morris and Samotij~\cite{BMS}, and by Saxton and Thomason~\cite{ST}.  Roughly speaking, it says that if a hypergraph $\cH$ has a somewhat uniform edge-distribution, then one can find a relatively small collection of sets covering all independent sets in $\cH$. Among others, this method provides an alternative proof of a recent breakthrough transference theorem of Conlon and Gowers~\cite{CG} and Schacht~\cite{Sch} for extremal results in sparse random setting. We refer the readers to~\cite{BMS, ST} for more details and applications, see also~\cite{BLST, BLST2} for more recent applications of container-type results in the arithmetic setting.

\medskip

Let $\cH$ be an $r$-uniform hypergraph with average degree $d$. For every $S\subseteq V(\cH)$, its co-degree, denoted by $d(S)$, is the number of edges in $\cH$ containing $S$, i.e.,~
$$d(S)=|\{e\in E(\cH): S\subseteq e\}|.$$ 
For every $j\in[r]$, denote by $\De_j$ the $j$-th maximum co-degree of $\cH$, i.e.,~
$$\De_j=\max\{d(S): S\subseteq V(\cH), |S|=j\}.$$
For any $\tau\in (0,1)$, define the following function which controls simultaneously the maximum co-degrees $\De_j$'s for all $j\in \{2,\ldots, r\}$: 
$$\De(\cH,\tau)=2^{{r\choose 2}-1}\sum_{j=2}^r2^{-{j-1\choose 2}}\frac{\De_j}{d\tau^{j-1}}.$$
Note that we are interested to have small $\tau$, as smaller $\tau$ means smaller family  of containers.
Here, as the codegrees are relatively small, only the $\Delta_r/(d\cdot\tau^{r-1})$ part matters.

We need the following version of the hypergraph container theorem (Corollary 3.6 in~\cite{ST}).
\begin{theorem}\label{thm-container}
Let $\cH$ be an $r$-uniform hypergraph on vertex set $[n]$. Let $0<\ep,\tau<1/2$. Suppose that $\tau<1/(200r!^2r)$ and $\De(\cH,\tau)\le \ep/(12r!)$. Then there exists $c=c(r)\le 1000r!^{3}r$ and a collection of vertex subsets $\cC$ such that

(i) every independent set in $\cH$ is a subset of some $A\in \cC$;

(ii) for every $A\in\cC$, $e(\cH[A])\le\ep e(\cH)$;

(iii) $\log|\cC|\le cn\tau\log(1/\ep)\log(1/\tau)$.
\end{theorem}


Given an integer $k\ge 3$, consider the $k$-uniform hypergraph $\cH_k$ encoding the set of all $k$-APs in $[n]$: $V(\cH_k)=[n]$ and the edge set of $\cH_k$ consists of all k-tuples that form a $k$-AP. It is easy to check that the number of $k$-APs in $[n]$ is $n^2/(2k)<e(\cH_k)<n^2/k$. Note that $\De_1\le k\cdot \frac{n}{k-1}<2n$ and 
\begin{eqnarray}\label{eq-kap-hyp}
d=d(\cH_k)\ge \frac{n}{2}, \quad \De_k=1, \quad \De_i\le\De_2\le {k\choose 2}<k^2  \ \mbox{ for $2\le i\le k-1$}.
\end{eqnarray}

Using the $k$-AP-hypergraph $\cH_k$, we obtain the following adaption of Theorem~\ref{thm-container} to the arithmetic setting.
\begin{cor}\label{cor-container}
Fix an arbitrary integer $k\ge 3$ and let $0<\ep,\tau<1/2$ be such that 
\begin{eqnarray}\label{eq-container}
\tau<1/(200k^{2k})\quad \mbox{ and }  \quad\ep n\tau^{k-1}>k^{3k}.
\end{eqnarray}
Then for sufficiently large $n$, there exists a collection $\cC$ of subsets of $[n]$ such that

(i) every $k$-AP-free subset of $[n]$ is contained in some $F\in\cC$;

(ii) for every $F\in\cC$, the number of $k$-APs in $F$ is at most $\ep n^2$;

(iii) $\log|\cC|\le 1000k^{3k}n\tau\log(1/\ep)\log(1/\tau)$.
\end{cor}
\begin{proof}
Consider the $k$-AP hypergraph $\cH_k$. Fix any $0<\ep,\tau<\frac{1}{2}$ such that $\tau<\frac{1}{200k^{2k}}<2^{-3k}$ and $\ep n\tau^{k-1}>k^{3k}$. Define $\alpha_j:=2^{-{j-1\choose 2}}\cdot\tau^{-(j-1)}$ for $2\le j\le k$. Since $\tau<2^{-3k}$, we have that for $2\le j\le k-2$,
\begin{eqnarray}\label{eq-alpha}
\frac{\alpha_j}{\alpha_{j+1}}=\frac{2^{{j\choose 2}}\cdot\tau^{j}}{2^{{j-1\choose 2}}\cdot\tau^{j-1}}=2^{j-1}\tau<2^k\tau<1 \quad\mbox{ and}\quad \frac{k^3\al_{k-1}}{\al_k}=k^32^{k-2}\tau<1.
\end{eqnarray}
Note that for any $k\ge 3$, we have that $\tau<1/(200k^{2k})<1/(200k!^2k)$. We now bound the function $\De(\cH_k,\tau)$ from above as follows:
\begin{eqnarray*}
\De(\cH_k,\tau)&=&2^{{k\choose 2}-1}\sum_{j=2}^k \al_j\frac{\De_j}{d}\overset{\eqref{eq-kap-hyp}}{\le} 2^{{k\choose 2}-1}\left(\sum_{j=2}^{k-1} \al_j\frac{k^2}{d}+\frac{\al_k}{d}\right)\overset{\eqref{eq-alpha}}{\le}2^{{k\choose 2}-1}\left((k-2)\al_{k-1}\frac{k^2}{d}+\frac{\al_{k}}{d}\right)\\
&\overset{\eqref{eq-alpha}}{\le}&2^{{k\choose 2}-1}\cdot \frac{2\al_k}{d}=\frac{2^{k-1}}{d\tau^{k-1}}\overset{\eqref{eq-kap-hyp}}{\le}\frac{2^k}{n\tau^{k-1}}\overset{\eqref{eq-container}}{\le} \frac{\ep}{12k!}.
\end{eqnarray*}

We  now apply Theorem~\ref{thm-container} on $\cH_k$ to obtain $\cC$. Then the conclusions follow from the observation that every independent set in $\cH_k$ is a $k$-AP-free subset of $[n]$.
\end{proof}
\subsection{Supersaturation}\label{sec-supsat}
In this subsection, we present the second main ingredient for the proof of Theorem~\ref{thm-main}: a supersaturation result, Lemma~\ref{lem-supsat}, which states that many $k$-APs start to appear in a set once its size is larger than $r_k(n)$. 

First notice that for any $A\subseteq [n]$ of size $K\cdot r_k(n)$, the following greedy algorithm gives 
\begin{eqnarray}\label{eq-greedy}
\Ga_k(A)\ge (K-1)\cdot r_k(n).
\end{eqnarray}
Set $B:=A$. Repeat the following process $(K-1)\cdot r_k(n)$ times: since $|B|>r_k(n)$, there is a $k$-AP in $B$; update $B$ by removing an arbitrary element in this $k$-AP. We  use a random sparsening trick to improve this simple argument.

\begin{lemma}\label{lem-supsat2}
For every $A\subseteq [n]$ of size $K\cdot r_k(n)$ with $K\ge 2$, we have
$$\Ga_k(A)\ge \left(\frac{K}{2}\right)^k\cdot r_k(n).$$
\end{lemma}
\begin{proof}
Let $T$ be a set chosen uniformly at random among all subsets of $A$ of size $2r_k(n)$. Then the expected number of $k$-APs in $T$ is 
$$\bE[\Ga_k(T)]=\frac{{|A|-k\choose |T|-k}}{{|A|\choose |T|}}\cdot \Ga_k(A)\le \left(\frac{|T|}{|A|}\right)^k\cdot \Ga_k(A)=\frac{\Ga_k(A)}{(K/2)^k}.$$
Thus, there exists a choice of $T$ such that $\Ga_k(T)\le \frac{\Ga_k(A)}{(K/2)^k}$. On the other hand, from~\eqref{eq-greedy}, $\Ga_k(T)\ge r_k(n)$, hence $\Ga_k(A)\ge \left(\frac{K}{2}\right)^k\cdot r_k(n)$ as desired.
\end{proof}

However, the bound given above is still linear in $|A|$, which is not sufficient for our purposes. A superlinear bound is provided in the following lemma, which implies that $\Ga_k(A)\ge |A|\cdot$poly$\log(n)$ for infinitely many values of $n$ (as in Theorem~\ref{thm-supsat}).  A key new idea in our proof is that an averaging argument is carried out only over a set of carefully chosen arithmetic progressions with prime common differences. To obtain a superlinear bound, we will apply the following lemma with roughly $M\sim |A|\cdot \left(\frac{|A|}{n}\right)^{k+1}$, and $|A|\ge r_k(n)$.

\begin{lemma}\label{lem-supsat}
For any $1\le M\le n$ and $A\subseteq [n]$, if ${|A|}/{M}$ is sufficiently large and ${|A|}/{n}\ge 8K\cdot{r_k(M)}/{M}$ with $K\ge 2$, then
$$\Ga_k(A)\ge\frac{|A|^2}{M^2}\cdot\frac{K^k\cdot r_k(M)}{2^{k+4}\log^2 n}.$$
\end{lemma}
\begin{proof}
Define $x=|A|/(4M)$, and assume that it is sufficiently large. Then the Prime Number Theorem (see e.g.~\cite{PNT}) implies that the number of prime numbers less than $x$ is at least $x/\log x$ and at most $2x/\log x$. Denote $\cB_d$  the set of $M$-term arithmetic progressions with common difference $d$ in $[n]$ and set
$$\cB:=\bigcup_{\substack{d \mbox{ is prime}\\
                                           d\le x}}
\cB_d,$$
that is, $\cB$ consists of all $M$-APs whose common difference is a prime number not larger than $x$. We notice first that any $k$-AP can occur in at most $M\log n$ many members of $\cB$. Indeed, fix an arbitrary $k$-AP, say $Q'$, with common difference $d'$. Note that every $M$-AP $Q$ containing $Q'$ can be constructed in two steps: 

(i) choose $1\le i\le M$ and set the $i$-th term of $Q$ to $\min(Q')$; 

(ii) choose the common difference $d$ for $Q$.\\
There are clearly at most $M$ choices for (i). As for (ii), in order to have $Q'\subseteq Q$, we need $d|d'$. Since $Q\in \cB$, $d$ must be a prime divisor of $d'$. Using that the number of prime divisors of $d'$ is at most $\log d'\le \log n$, the number of such choices is at most $\log n$. As a consequence, we have that 
\begin{eqnarray}\label{eq-oc}
\Ga_k(A)\ge \frac{1}{M\log n}\sum_{B\in\cB}\Ga_k(A\cap B).
\end{eqnarray}

Let $\cG\subseteq \cB$ consists of all $B\in\cB$ such that $|A\cap B|\ge K\cdot r_k(M)$. Then by Lemma~\ref{lem-supsat2}, we have $\Ga_k(A\cap B)\ge (K/2)^k\cdot r_k(M)$ for every $B\in\cG$. Together with~\eqref{eq-oc}, this gives that
\begin{eqnarray}\label{eq-needg}
\Ga_k(A)\ge \frac{1}{M\log n}\sum_{B\in\cG}\Ga_k(A\cap B)\ge |\cG|\cdot\frac{K^k\cdot r_k(M)}{2^kM\log n}.
\end{eqnarray}
Our next goal is to give a lower bound on $|\cG|$, to achieve this, we will do a double-counting on $\sum_{B\in\cB}|A\cap B|$. 

For each $d\le x$, define $I_d:=[(M-1)d+1, n-(M-1)d]$. Then every $z\in I_d$ appears in exactly $M$ members of $\cB_d$. Since $x=|A|/(4M)$, 
$$|A\cap I_d|=|A|-2(M-1)d\ge |A|-2Mx\ge \frac{|A|}{2}.$$ 
As an immediate consequence of the Prime Number Theorem, the number of primes less than $x$, which is the number of choices for $d$, is at least $x/\log x$ and at most $2x/\log x$ for sufficiently large $x$. Therefore,
\begin{eqnarray}\label{eq-int-upp}
\sum_{B\in\cB}|A\cap B|=\sum_{\substack{d \mbox{ is prime}\\d\le x}}\sum_{B\in \cB_d}|A\cap B|\ge M\sum_{\substack{d \mbox{ is prime}\\d\le x}}|A\cap I_d|\ge M\cdot \frac{x}{\log x}\cdot \frac{|A|}{2}.
\end{eqnarray}
On the other hand, since $|\cB_d|<n$, for each $d$ we have $|\cB|\le \frac{2x}{\log x}\cdot n$, hence
\begin{eqnarray}\label{eq-int-low}
\sum_{B\in\cB}|A\cap B|\le M|\cG|+K\cdot r_k(M)\cdot |\cB\setminus \cG|\le M|\cG|+K\cdot r_k(M)\cdot \frac{2xn}{\log x}.
\end{eqnarray}
Combining~\eqref{eq-int-upp} and~\eqref{eq-int-low}, we get
\begin{eqnarray*}
|\cG|&\ge& \frac{x}{\log x}\cdot \frac{|A|}{2}-K\cdot \frac{r_k(M)}{M}\cdot \frac{2xn}{\log x}=\frac{x}{\log x}\left(\frac{|A|}{2}-2K\cdot \frac{r_k(M)}{M}\cdot n\right)\\
&\ge&\frac{x}{\log n}\cdot \frac{|A|}{4}= \frac{|A|^2}{16M\log n},
\end{eqnarray*}
where the last inequality follows from $\frac{|A|}{n}\ge 8K\cdot\frac{r_k(M)}{M}$. Thus, by~\eqref{eq-needg}, we have
$$\Ga_k(A)\ge \frac{|A|^2}{16M\log n}\cdot \frac{K^k\cdot r_k(M)}{2^kM\log n}=\frac{|A|^2}{M^2}\cdot \frac{K^k\cdot r_k(M)}{2^{k+4}\log^2 n}.$$
\end{proof}

\section{Proof of Theorem~\ref{thm-main}}\label{sec-main}
Throughout this section, we fix $k$ a positive integer and write $r(n)$ instead of $r_k(n)$ and define $\ f(n)=r(n)/n$. We will use the following functions:
\begin{eqnarray}\label{eq-para-setup}
M(n)=\frac{n}{\log^{3k}n}\left(\frac{r(n)}{n}\right)^{k+2},\ \ep(n)=\frac{\log^{3k-2}n}{n}\left(\frac{n}{r(n)}\right)^{k-1}, \ \tau(n)=\frac{r(n)}{n}\frac{1}{\log^3n}.
\end{eqnarray} 
We first observe a simple fact about the function $r(n)$. Since the property of having no $k$-AP is invariant under translation, for any given $m<n$, if we divide $[n]$ into consecutive intervals of length $m$, then any given $k$-AP-free subset of $[n]$ contains at most $r(m)$ elements from each interval. Thus,
\begin{eqnarray}\label{eq-f-dec}
r(n)\le \left\lceil \frac{n}{m}\right\rceil\cdot  r(m).
\end{eqnarray} 
Since $\frac{1}{n}\cdot \left\lceil \frac{n}{m}\right\rceil<\frac{2}{m}$ for any $m<n$, dividing by $n$ on both sides of~\eqref{eq-f-dec} yields:
\begin{fact}\label{fact-dec}
For every $m<n$, $f(n)<2f(m)$.
\end{fact}

For the proof of Theorem~\ref{thm-supsat}, we will apply Lemma~\ref{lem-supsat} with $M$ defined as in~\eqref{eq-para-setup}. To do so, it requires that the function $r(n)$ is ``smooth''. This is proved in the following lemma. 

\begin{lemma}\label{lem-smooth2}
Given $k\ge 3$, there exists $C:=C(k)>4$ and an infinite sequence $\{n_i\}_{i=1}^{\infty}$, such that 
$$C\frac{r(n_i)}{n_i}\ge\frac{r(M(n_i))}{M(n_i)}$$
for all $i\ge 1$, where $M(n)$ is defined as in~\eqref{eq-para-setup}.
\end{lemma}

\begin{proof}
	Fix $C=C(k)>4$ a sufficiently large constant. From Behrend's construction, we know that $f(n)>2^{-5\sqrt{\log n}}$. We need to show that, for infinitely many $n$, $Cf(n)\ge f(M(n))=f\left(\frac{n}{\log^{3k}n}f(n)^{k+2}\right)$. Suppose to the contrary, that for all but finitely many $n$, $f(n)\le C^{-1}f(M(n))$. Let $n_0$ be the largest integer such that $f(n)> C^{-1}f(M(n))$. 
	
	Define a decreasing function $g(x)=2^{-(5k+11)\sqrt{\log x}}$ for $x\ge 1$. Note that for sufficiently large $n$, since $f(n)>2^{-5\sqrt{\log n}}$,
	$$M(n)=\frac{n}{\log^{3k}n}f(n)^{k+2}>\frac{n}{\log^{3k}n}\cdot 2^{-5(k+2)\sqrt{\log n}}>n\cdot 2^{-(5k+11)\sqrt{\log n}}=n\cdot g(n).$$ Then by Fact~\ref{fact-dec}, we have $f(M(n))<2f(n\cdot g(n))$. Therefore, by the definition of $n_0$, we have that for any $n>n_0$,
	\begin{eqnarray}\label{eq-base}
	f(n)\le C^{-1}f(M(n))<\left(\frac{C}{2}\right)^{-1}f(n\cdot g(n)).
	\end{eqnarray}
	
	Fix an integer $n>n_0^2$ and set  $t=\lfloor\frac{1}{2}\frac{\sqrt{\log n}}{5k+11}\rfloor$. We will show by induction that for every $1\le j\le t$, 
	\begin{eqnarray}\label{eq-iterate}
	f(n)<\left(\frac{C}{4}\right)^{-j}f(n\cdot g(n)^j).
	\end{eqnarray}
	The base case, $j=1$,  is given by~\eqref{eq-base}. Suppose~\eqref{eq-iterate} holds for some $1\le j<t$. Define $n':=n\cdot g(n)^j$. Then $$n'>n\cdot g(n)^t=n\cdot 2^{-(5k+11)\sqrt{\log n}\cdot \lfloor\frac{1}{2}\frac{\sqrt{\log n}}{5k+11}\rfloor}\ge n\cdot 2^{-\frac{1}{2}\log n}=\sqrt{n}>n_0.$$
	So by~\eqref{eq-base}, $f(n')<\left(\frac{C}{2}\right)^{-1}f(n'\cdot g(n'))$. Since $n'<n$ and $g(x)$ is decreasing, $n'\cdot g(n')>n'\cdot g(n)$. Then by Fact~\ref{fact-dec}, $f(n'\cdot g(n'))<2f(n'\cdot g(n))$. Hence, $f(n')<\left(\frac{C}{4}\right)^{-1}f(n'\cdot g(n))$. Thus by the induction hypothesis
	\begin{eqnarray*}
		f(n)&<& \left(\frac{C}{4}\right)^{-j}f(n\cdot g(n)^j)=\left(\frac{C}{4}\right)^{-j}f(n')\\
		&<&\left(\frac{C}{4}\right)^{-j}\left(\frac{C}{4}\right)^{-1}f(n'\cdot g(n))=\left(\frac{C}{4}\right)^{-(j+1)}f(n\cdot g(n)^{j+1}).
	\end{eqnarray*}
This proves~\eqref{eq-iterate} for $j=t$  and note that $f(n)\le 1$ and that $f(n\cdot g(n)^t)<2f(\sqrt{n})$ by Fact~\ref{fact-dec}, hence
	$$f(n)<\left(\frac{C}{4}\right)^{-t}f(n\cdot g(n)^t)<\left(\frac{C}{4}\right)^{-t}\cdot 2f(\sqrt{n})\le 2\left(\frac{C}{4}\right)^{-t}=2\left(\frac{C}{4}\right)^{-\lfloor\frac{1}{2}\frac{\sqrt{\log n}}{5k+11}\rfloor}<2^{-5\sqrt{\log n}}$$
	for $C$ sufficiently large, a contradiction.
\end{proof}

Theorem~\ref{thm-supsat} follows immediately from Lemmas~\ref{lem-supsat} and~\ref{lem-smooth2}.
\begin{proof}[Proof of Theorem~\ref{thm-supsat}]
Let $K$ be the constant from Lemma~\ref{lem-supsat}. Let $C$ be the constant and $\{n_i\}_{i=1}^{\infty}$ be the sequence from Lemma~\ref{lem-smooth2}. Define $C'=8CK$. Fix an arbitrary $n\in\{n_i\}_{i=1}^{\infty}$ and write $M=M(n)$ as defined in~\eqref{eq-para-setup}. Let $A\subseteq [n]$ be an arbitrary set of size $C'r(n)$. Then by Lemma~\ref{lem-smooth2}, 
$$\frac{|A|}{n}=\frac{8CK\cdot r(n)}{n}\ge 8K\frac{r(M(n))}{M(n)}.$$ 
By Fact~\ref{fact-dec}, $\frac{2r(M)}{M}>\frac{r(n)}{n}$. Thus by Lemma~\ref{lem-supsat} and that $K\ge 2$, $C>4$, we have 
\begin{eqnarray*}
\Ga_k(A)&>&\frac{|A|^2}{M^2}\cdot\frac{K^k\cdot r(M)}{2^{k+4}\log^2 n}=\frac{(8CK)^2r(n)^2}{M^2}\cdot\frac{K^k\cdot r(M)}{2^{k+4}\log^2 n}=\frac{r(n)^2}{M\log^2 n}\cdot \frac{2r(M)}{M}\cdot 
\frac{(8CK)^2K^k}{2^{k+5}}\\
&>&\frac{r(n)^2}{M\log^2 n}\cdot \frac{2r(M)}{M}>\frac{r(n)^2}{M\log^2 n}\cdot \frac{r(n)}{n}=\log^{3k-2}n\left(\frac{n}{r(n)}\right)^{k-1}n.
\end{eqnarray*}
\end{proof}

\begin{proof}[Proof of Theorem~\ref{thm-main}]
Let $\{n_i\}_{i=1}^{\infty}$ be the infinite sequence guaranteed by Lemma~\ref{lem-smooth2}. We will show that the conclusion holds for this sequence of values of $n$. Let $M=M(n),\ep=\ep(n)$ and $\tau=\tau(n)$ be as defined in~\eqref{eq-para-setup}. For sufficiently large $n$, we have that $\tau<\frac{1}{200k^{2k}}$ and
$$\ep n\tau^{k-1}=\frac{\log^{3k-2}n}{n}\left(\frac{n}{r(n)}\right)^{k-1}\cdot n\cdot \left(\frac{r(n)}{n}\frac{1}{\log^3n}\right)^{k-1}=\log n>k^{3k}.$$
Thus by Corollary~\ref{cor-container}, there is a family $\cC$ of containers such that every $k$-AP-free subset of $[n]$ is a subset of some container in $\cC$. By~\eqref{eq-para-setup}, $\log\frac{1}{\ep}\log\frac{1}{\tau}<\log^2n$, thus 
$$\log|\cC|\le 1000k^{3k}n\tau\log\frac{1}{\ep}\log\frac{1}{\tau}< 1000k^{3k}n\cdot\frac{r(n)}{n}\frac{1}{\log^3n}\cdot\log^2n=o(r(n)).$$

Since for every container $A\in\cC$, the number of $k$-APs in $A$ is at most $\ep n^2$, then by Theorem~\ref{thm-supsat}, $|A|< C'r(n)$ for every $A\in\cC$. Recall that every $k$-AP-free subset is contained in some member of $\cC$. Hence, the number of $k$-AP-free subsets of $[n]$ is at most
$$\sum_{A\in\cC}2^{|A|}\le |\cC|\cdot \max_{A\in\cC}2^{|A|}<2^{o(r(n))}\cdot 2^{C'r(n)}=2^{O(r(n))}.$$
\end{proof}

\begin{proof}[Proof of Corollary~\ref{cor-dense}]
Let $\{n_i\}_{i=1}^{\infty}$ be a sequence of integers for which the conclusion of Theorem~\ref{thm-main} holds. Fix an arbitrary $\ep>0$ and $n_i$. From Theorem~\ref{thm-main}, we know that the number of $k$-AP-free subsets of $[n_i]$ is at most $2^{c\cdot r(n_i)}$ for some absolute constant $c>0$. For any $\ep n_i\le m<n_i$, by~\eqref{eq-f-dec}, we have that $r(n_i)\le\left\lceil \frac{1}{\ep}\right\rceil\cdot  r(m)<\frac{2}{\ep}\cdot r(m)$. Therefore, by setting $b=2c/\ep$, we have that the number of $k$-AP-free subsets of $[m]$ is at most $2^{c\cdot r(n_i)}\le 2^{b\cdot r(m)}$. It then follows that $m\in A(b)$ for any $\ep n_i\le m<n_i$ and that $|A(b)\cap [n_i]|/n_i\ge1-\ep$ as desired. 
\end{proof}

The proof of Theorem~\ref{prop-cond} is along the same line as of the proof of Theorem~\ref{thm-main}, hence we provide here only a sketch of it. The difference is that to prove this weaker bound, we only need supersaturation results for sets of size $n/$poly$\log n$. For sets of this size, we do not need the technical condition in Lemma~\ref{lem-smooth2} and we can invoke Lemma~\ref{lem-supsat} with $M=n^{o(1)}$ for all values of $n$ instead of $M=n^{1-o(1)}$ as in proof of Theorem~\ref{thm-main}.
\begin{proof}[Proof of Theorem~\ref{prop-cond}]
Fix an arbitrary $0<\ga<1$. We apply Corollary~\ref{cor-container} with $\ep=n^{-\ga/2}$, $\tau=n^{-\frac{1}{k-1}+\ga/2}$ and let $\cC$ be the family of containers of size $\log|\cC|=o(n^{1-\frac{1}{k-1}+\ga})$. Each container contains at most $\ep n^2=n^{2-\ga/2}$ many $k$-APs. It follows that for every $A\in\cC$, $|A|\le\frac{C'n}{h(n)}$ for some $C'=C'(k,c,\ga)$, since otherwise applying Lemma~\ref{lem-supsat} on $A$ with $M=n^{\ga/4}$ would imply $\Ga_k(A)>n^{2-\ga/3}>\ep n^2$, a contradiction. Thus, the number of $k$-AP-free subsets of $[n]$ is at most $|\cC|\cdot 2^{C'n/h(n)}=2^{2C'n/h(n)}$, as desired. Similarly, the number of $k$-AP-free $m$-subsets of $[n]$ is at most $|\cC|\cdot {C'n/h(n)\choose m}\le 2^m\cdot {C'n/h(n)\choose m}\le {2C'n/h(n)\choose m}$, where the first inequality follows from $m\ge n^{1-\frac{1}{k-1}+\ga}\ge \log|\cC|$.
\end{proof}

\begin{proof}[Proof of Theorem~\ref{thm-prime-density}]
Fix an arbitrary $\de>0$ and an integer $k\ge 3$. Let $S\subseteq \pri$ be a random subset of $\pri$, in which each element is chosen with probability $p=n^{-1/k}$ independently of others. A standard application of Chernoff bound implies that $|S|\ge p|\pri|/2$ with probability $1-o(1)$.

Set $\beta=\de/50, \ga=1/(10k^2)$ and $m=\de p|\pri|/5$. Then by the Prime Number Theorem, we have
$$m=\Omega\left(\frac{n^{1-1/k}}{\log n}\right)>n^{1-\frac{1}{k-1}+\ga}.$$
Thus by Theorem~\ref{thm-prime-counting} the number of $k$-AP-free $m$-subsets in $\pri$ is at most ${\beta\pri\choose m}$.

Let $X$ be the number of $k$-AP-free $m$-subsets in $S$, and $Y$ be the number of $(k+1)$-APs in $S$. 
Then,
$$\bE[X]\le {\beta|\pri|\choose m}p^m\le \left(\frac{e\cdot \beta|\pri|}{m}\cdot p\right)^m=\left(\frac{e\cdot(\de/50)|\pri|\cdot p}{\de p|\pri|/5}\right)^m=\left(\frac{e}{10}\right)^m=o(1).$$
Thus by Markov's inequality, $X=0$ with probability at least $2/3$.

By the Green-Tao theorem, the number of $(k+1)$-APs in $\pri$ is $\Theta\left(\frac{n^2}{(\log n)^{k+1}}\right)$. Thus,
$$\bE[Y]\le \frac{n^2}{\log^kn}\cdot p^{k+1}=\frac{n^{1-1/k}}{\log^k n}.$$
We have, by Markov's inequality, that $Y\le \frac{3n^{1-1/k}}{\log^kn}$ with probability at least $2/3$. Therefore, with positive probability, there is a choice of $S$ such that $|S|\ge p|\pri|/2$, $X=0$ and $Y\le \frac{3n^{1-1/k}}{\log^kn}$. Let $S'$ be the set obtained from $S$ by deleting one element from every $(k+1)$-AP in $S$. Then 
$$|S'|=|S|-Y\ge \frac{p|\pri|}{2}-\frac{3n^{1-1/k}}{\log^kn}\ge \frac{n^{1-1/k}}{3\log n}-\frac{3n^{1-1/k}}{\log^kn}\ge \frac{n^{1-1/k}}{4\log n}.$$
We claim that $S'$ has the desired property. Indeed, clearly $S'$ is $(k+1)$-AP-free. Suppose $S'$ is not $(\de, k)$-Szemer\'edi, then there exists a $k$-AP-free subset $U\subseteq S'$ of size 
$$|U|\ge\de|S'|\ge \de\cdot \frac{n^{1-1/k}}{4\log n}> \frac{\de p|\pri|}{5}=m.$$
However, this contradicts that $X=0$.
\end{proof}

\section*{Acknowledgement}
We would like to thank J\'ozsef Solymosi who suggested considering the set of primes instead of integers in Theorem~\ref{thm-prime-density}. We also thank the anonymous referee for the helpful comments which greatly improved the presentation of this paper.

\section{Appendix}
The idea behind the container method is to have a small certificate set to describe a large independent set. We give two further applications using this idea.

\subsection{A variation of van der Waerden's theorem}\label{sec-1-3}
Van der Waerden's theorem~\cite{vdW}, a classical result in Ramsey theory, says that the set of integers is rich in arithmetic progressions: for any positive integers $k$ and $r$, there exists $n_0>0$ such that every $r$-coloring of $[n]$ with $n>n_0$ yields a monochromatic $k$-term arithmetic progression. Denote by $W(k;r)$ the $r$-colored van der Waerden number, i.e.,~the minimum integer $n$ such that every $r$-coloring of $[n]$ contains a monochromatic $k$-AP. The best known upper bound on $W(k;r)$ is due to   Gowers~\cite{Gow2}: $W(k;r)\le 2^{2^{r^{2^{2^{k+9}}}}}$. For two colors, the best lower bound is due to Berlekamp~\cite{Ber}: $W(p+1;2)\ge p2^{p}$, where $p$ is a prime number.

By setting $\de=1/r$ in Theorem~\ref{thm-prime-density}, we immediately obtain the following extension of results of Spencer~\cite{Sp}, Ne\v{s}et\v{r}il and R\"odl~\cite{NR-vdw} in primes on
restricted van der Waerden's theorem.

\begin{cor}\label{thm-prime}
	For any $r\ge 2$ and $k\ge 3$, there exists a set of primes $S\subseteq \pri$ such that $S$ is $(k+1)$-AP-free and any $r$-coloring of $S$ yields a monochromatic $k$-AP.
\end{cor}

This type of question was first raised by Erd\H{o}s and Hajnal~\cite{EH}: They asked whether there exists a $K_{k+1}$-free graph such that every $r$-edge-coloring of it induces a monochromatic $K_k$. This was answered in the affirmative by Folkman~\cite{Fo} for $r=2$, and by Ne{\v{s}}et{\v{r}}il and R{\"o}dl~\cite{NR} for arbitrary $r$, see also R{\"o}dl-Ruci\'{n}ski-Schacht~\cite{RRS} for recent developments.

We will use the hypergraph container method to give an alternative proof of Corollary~\ref{thm-prime} without invoking Theorem~\ref{thm-prime-density}. This proof draws on ideas from Nenadov-Steger~\cite{NS} and R{\"o}dl-Ruci\'{n}ski-Schacht~\cite{RRS}. 

If we work in the set of all integers instead of just the primes, we are able to get the following quantitative result, Proposition~\ref{thm-folkman}.\footnote{We remark that a very recent paper~\cite{high-girth} achieves essentially the same quantitative bound, but with a stronger ``large girth'' property.} A set $S\subseteq [n]$ is $(k;r)$-\emph{Folkman} if $S$ is $(k+1)$-AP-free and every $r$-coloring of $S$ contains a monochromatic $k$-AP. Define
$$g(k;r):=\min\{n: \exists S\subseteq [n], \ S\mbox{ is $(k;r)$-Folkman}\}.$$
Clearly $g(k;r)\ge W(k;r)$. To see this, simply notice that if $n=W(k;r)-1$, then there exists an $r$-coloring of $[n]$ with no monochromatic $k$-AP and any $(k+1)$-AP-free subset of it inherits this property, implying that $g(k;r)>n$.
\begin{prop}\label{thm-folkman}
	For any $r\ge 2$ and $k\ge 40$,
	$$g(k;r)\le k^{4k^3}W(k;r)^{5k^2}.$$
\end{prop}
The proof of Proposition~\ref{thm-folkman} follows along the same line as Corollary~\ref{thm-prime}. We omit its proof.

For the proof of Corollary~\ref{thm-prime}, we need the following supersaturation lemma. Given a coloring $\phi$, denote by $\phi^{(i)}:=\phi^{-1}(i)$ the $i$-th color class. We write $W:=W(k;r)$.

\begin{lemma}\label{lem-color-prime}
	Given any coloring $\phi: \pri\rightarrow [r+1]$ with $n$ sufficiently large, if $\sum_{i\le r}\Gamma_k(\phi^{(i)})\le n^2/(\log n)^{W+1}$, then $|\phi^{(r+1)}|\ge n/(\log n)^{W+1}$.
\end{lemma}
\begin{proof}
	Fix an arbitrary $(r+1)$-coloring $\phi$ of $\pri$ such that $\sum_{i\le r}\Gamma_k(\phi^{(i)})\le n^2/(\log n)^{W+1}$. Recall that the number of $W$-term arithmetic progressions in $[n]$ is at least $c_Wn^2/(\log n)^{W}$ for some constant $c_W>0$. Let $x\cdot c_Wn^2/(\log n)^{W}$ be the number of $W$-APs colored completely by one of the first $r$ colors. Then by the definition of $W(k;r)$, each of these $W$-APs induces a monochromatic $k$-AP in the first $r$ colors. We claim that every $k$-AP is contained in at most $W^2$ many $W$-APs. Indeed, given any $k$-AP $\{a_1,\ldots, a_k\}$, a $W$-AP $\{b_1,\ldots, b_W\}$ containing it will be uniquely determined once we fix $a_1=b_i$, $a_2=b_j$ for $1\le i\neq j\le W$. The number of choices of the two indices $i$ and $j$ is at most $W^2$. Therefore, the number of monochromatic $k$-APs in the first $r$ colors is at least $(xc_Wn^2/(\log n)^{W})/W^2$. On the other hand, this number is at most $n^2/(\log n)^{W+1}$, thus for sufficiently large $n$ we have
	$$\frac{xc_Wn^2}{(\log n)^{W}W^2}\le \frac{n^2}{(\log n)^{W+1}} \quad\Rightarrow  \quad x\le \frac{W^2}{c_W\log n}\le \frac{1}{2}.$$ 
	So the number of $W$-APs containing at least one element from $\phi^{(r+1)}$ is at least $c_Wn^2/2(\log n)^{W}$. Note that each element in $[n]$ is in at most $W\cdot \frac{n}{W-1}<2n$ many $W$-APs. Indeed, there are at most $W$ choices to decide which term in a $W$-AP an element in $[n]$ will be and at most $\frac{n}{W-1}$ many choices to choose the common difference. Therefore, $|\phi^{(r+1)}|\ge (c_Wn^2/2(\log n)^{W})/2n\ge n/(\log n)^{W+1}$, as desired.
\end{proof}

\begin{proof}[Proof of Corollary~\ref{thm-prime}]
	We set the parameters as follows:
	\begin{eqnarray}\label{eq-param-prime}
		p=\frac{1}{n^{1/k}(\log n)^W}, \quad\quad \ep=\frac{1}{(\log n)^{2W}} ,\quad\quad \tau=\frac{1}{n^{1/k}(\log n)^{3W}}. 
	\end{eqnarray}
	For any $k\ge 3$, and $p,\ep,\tau$ defined in~\eqref{eq-param-prime}, and sufficiently large $n$, we have that 
	$$1000rk^{3k}\log\frac{1}{\tau}\log\frac{1}{\ep}=1000rk^{3k}\cdot \left(\frac{\log n}{k}+3W\log\log n\right)\cdot 2W\log\log n\le \log^2 n.$$
	Thus, we have
	\begin{eqnarray}\label{eq-fi1-prime}
		\frac{np}{10(\log n)^{W+1}}=\frac{n^{1-1/k}}{10(\log n)^{2W+1}}\ge \frac{n^{1-1/k}}{(\log n)^{3W-2}}= n\tau\cdot\log^2n\ge 1000rk^{3k}n\tau\log\frac{1}{\tau}\log\frac{1}{\ep}.
	\end{eqnarray}
	We also need the following inequality.
	\begin{eqnarray}\label{eq-fi2-prime}
		\frac{np}{10(\log n)^{W+1}}=\frac{n^{1-1/k}}{10(\log n)^{2W+1}}\ge \frac{n^{1-1/k}}{(\log n)^{k+(k+1)W}}=p^{k+1}\frac{n^2}{(\log n)^{k}}.
	\end{eqnarray}
	
	Let $S\subseteq \pri$ be a random subset of $\pri$, in which every element is chosen with probability $p$ as defined in~\eqref{eq-param-prime}, independently of each other. Denote by $B_1$ the event that every $r$-coloring of $S$ contains a monochromatic $k$-AP and by $B_2$ the event that $S$ is $(k+1)$-AP-free. We will show that $\bP[B_1]+\bP[B_2]>1$, which then implies that with positive probability there is a choice of $S\subseteq \pri$ with the desired properties.
	
	To estimate $\bP[B_2]$, we apply the FKG inequality (see e.g.~\cite{Bo}). Note that the number of $(k+1)$-APs in $\pri$ is at most $Cn^2/(\log n)^{k+1}$ for some constant $C>0$ depending only on $k$.
	\begin{eqnarray}\label{eq-APfree}
		\bP[B_2]\ge \left(1-p^{k+1}\right)^{Cn^2/(\log n)^{k+1}}>\exp\left\{-p^{k+1}\cdot \frac{n^2}{(\log n)^k}\right\}.
	\end{eqnarray}
	
	Let $\ep,\tau$ be as defined in~\eqref{eq-param-prime}. We claim that $\ep,\tau$ satisfy~\eqref{eq-container}. It follows immediately from the definition of $\tau$ that $\tau<\frac{1}{200k^{2k}}$. For the other inequality, 
	\begin{eqnarray*}
		\ep n\tau^{k-1}&=&\frac{n}{(\log n)^{2W}}\cdot\frac{1}{n^{\frac{k-1}{k}}(\log n)^{3(k-1)W}}= \frac{n^{1/k}}{(\log n)^{(3k-1)W}}\ge k^{3k}.
	\end{eqnarray*}
	Thus we can apply Corollary~\ref{cor-container} with $\ep, \tau$  to obtain $\cC$, the set of containers.
	
	
	We now bound $\bP[\overline{B_1}]$ from above. Note that $\overline{B_1}$ implies that there is an $r$-coloring of $S$ with no monochromatic $k$-AP. The idea is that if such a coloring exists, then necessarily there is a fairly dense set disjoint from the random set $S$, which is highly unlikely. Fix one such coloring $\sigma: S\rightarrow [r]$ with color classes $X_1,\ldots, X_r$. Since every $X_i$ is $k$-AP-free, $X_i\subseteq F_i$ for some container $F_i\in\cC$. Define $T=\pri\setminus \bigcup_i F_i$, so $S\cap T=\emptyset$. Notice that $T$ is independent of the initial coloring $\sigma$, and the number of choices for $T$ is at most $|\cC|^r$.
	
	We claim that the set $T$ obtained above is large: $|T|\ge n/(\log n)^{W-1}$. To see this, define an auxiliary $(r+1)$-coloring $\phi: \pri\rightarrow [r+1]$ as follows: 
	$\forall x\in T$, $\phi(x)=r+1$, and $\forall x\not\in T$, $\phi(x)=\min\{i: x\in F_i\}$. By Corollary~\ref{cor-container} (ii), the number of monochromatic $k$-APs in the first $r$ colors of $\phi$ is at most $r\cdot\ep n^2=rn^2/(\log n)^{2W}\le n^2/(\log n)^{W+1}$. Then by Lemma~\ref{lem-color-prime}, $|T|\ge n/(\log n)^{W-1}$ as desired.
	
	Applying the union bound over all possible choices of $T$, we obtain that
	\begin{eqnarray*}
		\bP[\overline{B_1}]&\le&\bP[\cup_T S\cap T=\emptyset]\le |\cC|^r\cdot (1-p)^{n/(\log n)^{W+1}}\\
		&\le&\exp\left\{r\cdot 1000k^{3k}n\tau\log\frac{1}{\tau}\log\frac{1}{\ep}-\frac{np}{(\log n)^{W+1}}\right\}\\
		&\overset{\eqref{eq-fi1-prime}}{\le}&\exp\left\{-\frac{np}{2(\log n)^{W+1}}\right\}\overset{\eqref{eq-fi2-prime}}{\le}\exp\left\{-p^{k+1}\frac{n^2}{(\log n)^{k}}\right\}<\bP[B_2].
	\end{eqnarray*}
	Thus we have $\bP[B_1]+\bP[B_2]=1-\bP[\overline{B_1}]+\bP[B_2]>1$ as desired.
\end{proof}


\subsection{$d$-fold sumset}\label{sec-1-4}
Denote by $dA$ the $d$-fold sumset of $A$: $dA=A+\ldots +A$, where $A\subseteq \mathbb{Z}_p$. How large does a set in $\mathbb{Z}_p$ have to be so that it is a $d$-fold sumset? Define $f_d(p)$ to be the maximum integer $\ell$ such that for any set $F$ of size $\ell$, $\mathbb{Z}_p-F=dA$ for some $A\subseteq \mathbb{Z}_p$. Green and Gowers~\cite{G}, using the discrete Fourier method, showed that $\Omega(\log p)=f_2(p)=O(p^{2/3}\log^{1/3}p)$. It was later improved by Alon~\cite{A} using eigenvalues of the Cayley sum-graphs:
$$\Omega\left(\frac{\sqrt{p}}{\sqrt{\log p}}\right)=f_2(p)=O\left(\frac{p^{2/3}}{\log^{1/3}p}\right).$$
It remains a difficult open question to close the gap above. In this section, we investigate this function for the $d$-fold sumset for every $d\ge 2$. Green and Gowers' proof in fact works for all $d\ge 2$, giving an upper bound\footnote{Perhaps a more involved argument using the Fourier technique can give improvement on this bound obtained directly from their argument.} $$f_d(p)=O(p^{2/3}\log^{1/3}p).$$

Our next result gives an improvement when $d\ge 3$.
\begin{theorem}\label{thm-dfold}
	For any $d\ge 2$,
	$$f_d(p)=O(p^{\frac{d}{2d-1}+o(1)}).$$
\end{theorem}

Here, to prove Theorem~\ref{thm-dfold}, we use Proposition~\ref{prop-random} instead to find such a small certificate. Though our bound for $d=2$ is weaker than the previous bounds by a polylog factor, it easily works for any $d\ge 2$. We think that the bound above is far from best possible. We conjecture that $f_d(p)=p^{c_d+o(1)}$, where $c_d\rightarrow 0$ as $d\rightarrow \infty$.

First, we need to define an auxiliary hypergraph, from which the upper bound on $f_d(p)$ can be derived. Given a set $T\subseteq \mathbb{Z}_p$, we define the $d$-uniform Cayley sum-hypergraph $G(\mathbb{Z}_p,T)$ generated by $T$ as follows: $V(G(\mathbb{Z}_p,T))=\mathbb{Z}_p$ and its edge set consists of all $d$-tuples $\{x_1,\ldots,x_d\}$ such that $x_1+\ldots+x_d=t$ for some $t\in T$. The following claim, due to Green~\cite{G} and Alon~\cite{A}, gives a way to obtain upper bound, we repeat here their short proof.

\begin{claim}\label{claim-obs}
	If there exists a set $T$ of size $t$ such that $t>2\alpha(G(\mathbb{Z}_p,T))$, then $f_d(p)\le 2t$.
\end{claim}
\begin{proof}
	If we have a set $F\subseteq \mathbb{Z}_p$ such that $\mathbb{Z}_p-F$ is not a sumset, then $f_d(p)\le |F|$. We will find such a set $F=T\cup T'$ in two steps with $|T|=|T'|=t$, where $T$ is the set guaranteed by the hypothesis, i.e., $t>2\alpha(G(\mathbb{Z}_p,T))$. Notice that if any set $S\subseteq \mathbb{Z}_p-T$ is a $d$-fold sumset for some set $A\subseteq \mathbb{Z}_p$, then $A$ has to be an independent set in $G(\mathbb{Z}_p,T)$. Note also that the number of $d$-fold sumsets $S$ cannot be larger than the number of sets $A$ that we generate $S=dA$ from. We then choose another set $T'$ of size $t$, there are ${p-t \choose t}$ many choices. Suppose that for each of these choices, $S=\mathbb{Z}_p-T-T'$ is a $d$-fold sumset $dA$, then from the observation above we have that the number of choices for $A$ is at least 
	$${p-t\choose t}> {p\choose \alpha(G(\mathbb{Z}_p,T))}.$$ 
	This is impossible since $A$ is an independent set in $G(\mathbb{Z}_p,T)$. Thus there exists a $T'$ such that $F=T\cup T'$ is the desired set.
\end{proof}

We use the following slight variation of Proposition 19 in~\cite{G-count-ss}. 
\begin{prop}\label{prop-random}
	For any $d\ge 2$, there exists $c:=c(d)$ such that the following holds. For any set $S\subseteq \mathbb{Z}_p$ of size $m$ with $m$ sufficiently large, there exists a set $R\subseteq S$ with $|R|\le m^{1/d}/c$ and $|\widehat{d}R|\ge cm$, where $\widehat{d}R:=\{x: x=a_1+\ldots+a_d \mbox{ with distinct } a_i\in R\}$.
\end{prop}
\begin{proof}[Proof of Theorem~\ref{thm-dfold}]
	Let $T$ be a random $t$-set of $\mathbb{Z}_p$ for $t=\frac{2}{c^2}(p\log p)^{d/(2d-1)}$. Let $G:=G(\mathbb{Z}_p, T)$. We will show that with high probability $\alpha(G)<t/2$, then the bound on $f_d(p)$ follows from Claim~\ref{claim-obs}.
	
	For every set $S\subseteq G$ of size $m=t/2$, by Proposition~\ref{prop-random}, $S$ contains a set $R$ of size at most $m^{1/d}/c$ and $|\widehat{d}R|\ge cm$. Clearly, $\widehat{d}R\subseteq dS$. It then follows that if every $R$ of size at most $m^{1/d}/c$ with $|\widehat{d}R|\ge cm$ is not an independent set, then $\al(G)< m$. Fix such a choice of $R$, then $R$ is an independent set only when $\widehat{d}R\cap T=\emptyset$. Thus the probability that $R$ is independent is at most $(1-\frac{cm}{p})^t\le e^{-cmt/p}$. Applying the union bound, we obtain that the probability that there exists an $R$ that is independent is at most
	$$\sum_{i=1}^{m^{1/d}/c}{p\choose i}\cdot e^{-cmt/p}\le \exp\left\{\frac{1}{c}t^{1/d}\log p-\frac{ct^2}{2p}\right\}=o(1),$$
	where the last equality follows from $t=\frac{2}{c^2}(p\log p)^{d/(2d-1)}$. Thus $\al(G)< t/2$ as desired.
\end{proof}

Here is a related conjecture of Alon~\cite{A}. 
\begin{conj}\label{conj-alon}
	There exist constants $c_1, c_2$ so that the following holds. Let $\Ga$ be an abelian group of odd order $n$. Then for every $1\le t\le n$, there is a subset $T\subseteq \Ga$ of size $t$ so that for $G:=G(\Ga, T)$, 
	$$\al(G)\le c_1\frac{n}{t}(\log n)^{c_2}.$$
\end{conj}
The conjecture was known to be true for $t=\Om(n)$. Alon~\cite{A} gave a bound $\al(G)\le \frac{n}{t^{1/2}}\log n$. Using the idea in the proof of Theorem~\ref{thm-dfold}, we can establish the following bound:  
$$\al(G)\le \left(\frac{n}{t}\right)^2(\log n)^2,$$
which implies Conjecture~\ref{conj-alon} for $t\ge n/(\log n)^{O(1)}$. 


\end{document}